\documentclass[11pt]{article}
\usepackage{amssymb,amsmath}
\usepackage[mathscr]{eucal}
\usepackage[cm]{fullpage}
\usepackage[english]{babel}
\usepackage[latin1]{inputenc}
\usepackage{xcolor}
\def\dom{\mathop{\mathrm{Dom}}\nolimits}
\def\im{\mathop{\mathrm{Im}}\nolimits}
\def\rank{\mathop{\mathrm{rank}}\nolimits}
\def\d{\mathrm{d}}
\def\id{\mathrm{id}}
\def\N{\mathbb N}
\def\PT{\mathcal{PT}}
\def\T{\mathcal{T}}
\def\Sym{\mathcal{S}}
\def\DP{\mathcal{DP}}
\def\A{\mathcal{A}}
\def\B{\mathcal{B}}
\def\C{\mathcal{C}}
\def\D{\mathcal{D}}
\def\ODP{\mathcal{ODP}}
\def\DI{\mathcal{DI}}
\def\OPDI{\mathcal{OPDI}}
\def\ODI{\mathcal{ODI}}

\def\MDI{\mathcal{MDI}}
\def\PODI{\mathcal{PODI}}
\def\POI{\mathcal{POI}}
\def\POPI{\mathcal{POPI}}
\def\PORI{\mathcal{PORI}}
\def\PO{\mathcal{PO}}

\def\POD{\mathcal{POD}}
\def\POP{\mathcal{POP}}
\def\POR{\mathcal{POR}}
\def\I{\mathcal{I}}
\newcommand{\trans}[1]{\left(\begin{smallmatrix} #1 \end{smallmatrix}\right)}
\newtheorem{theorem}{Theorem}[section]
\newtheorem{proposition}[theorem]{Proposition}

\newtheorem{lemma}[theorem]{Lemma}
\newtheorem{example}[theorem]{Example}
\newenvironment{proof}{\begin{trivlist}\item[\hskip%
\labelsep{\bf Proof.}]}%
{\qed\rm\end{trivlist}}
\newcommand{\qed}{{\unskip\nobreak
\hfil\penalty50\hskip .001pt \hbox{}
          \nobreak\hfil
         \vrule height 1.2ex width 1.1ex depth -.1ex
           \parfillskip=0pt\finalhyphendemerits=0\medbreak}}

\newcommand{\lastpage}{\addresss}

\newcommand{\addresss}{\small \sf

\noindent{\sc Ilinka Dimitrova},
Department of Mathematics,
Faculty of Mathematics and Natural Science,
South-West University "Neofit Rilski",
2700 Blagoevgrad,
Bulgaria;
e-mail: ilinka\_dimitrova@swu.bg.

\medskip

\noindent{\sc V\'\i tor H. Fernandes},
Center for Mathematics and Applications (NovaMath)
and Department of Mathematics, FCT NOVA,
Faculdade de Ci\^encias e Tecnologia,
Universidade Nova de Lisboa,
Monte da Caparica,
2829-516 Caparica,
Portugal;
e-mail: vhf@fct.unl.pt.

\medskip

\noindent{\sc J\"{o}rg Koppitz},
Institute of Mathematics and Informatics,
Bulgarian Academy of Sciences,
1113 Sofia,
Bulgaria;
e-mail: koppitz@math.bas.bg.

\medskip

\noindent{\sc Teresa M. Quinteiro},
Instituto Superior de Engenharia de Lisboa,
1950-062 Lisboa,
Portugal.
Also:
Center for Mathematics and Applications (NovaMath),
Faculdade de Ci\^encias e Tecnologia,
Universidade Nova de Lisboa,
Monte da Caparica,
2829-516 Caparica,
Portugal;
e-mail: tmelo@adm.isel.pt.
}

\title{On three submonoids of
        the dihedral inverse monoid on a finite set}

\author{I. Dimitrova, V\'\i tor H. Fernandes, J. Koppitz and T.M. Quinteiro.}

\begin{document}

\maketitle

\begin{abstract}
In this paper we consider three submonoids of the dihedral inverse monoid $\DI_n$,
namely its submonoids $\OPDI_n$, $\MDI_n$ and $\ODI_n$
of all orientation-preserving, monotone and order-preserving transformations, respectively.
For each of these three monoids, we compute the cardinality, give descriptions of Green's relations and determine the rank.
\end{abstract}

\medskip



\section{Introduction and Preliminaries}\label{presection}

Let $\Omega$ be a set and let $A\subseteq \Omega$. A mapping $\alpha:A\rightarrow \Omega$ is called a \textit{partial transformation} of $\Omega$.
We denote by $\dom(\alpha)$ and $\im (\alpha)$ the domain and the image (range) of $\alpha$, respectively. The natural number
$\rank(\alpha) = |\im(\alpha)|$ is called the rank of $\alpha$.
Clearly, $A=\dom(\alpha)$. If $A=\Omega$ then $\alpha$ is called a \textit{full transformation}. If $A=\emptyset$ then $\alpha$ is called the empty transformation and denoted by $\emptyset$.
Given partial transformations $\alpha$ and $\beta$, the composition $\alpha\beta$ is the partial transformation defined by $x(\alpha\beta)=(x\alpha)\beta$ for all $x\in \dom(\alpha\beta)=(\im(\alpha)\cap \dom(\beta))\alpha^{-1} = \{x \in \dom(\alpha) \mid x\alpha \in \dom(\beta)\}$. Observe that $\im(\alpha\beta)=(\im(\alpha)\cap\dom(\beta))\beta$.
Denote by $\PT(\Omega)$ the monoid (under composition) of all
partial transformations on $\Omega$, by $\T(\Omega)$ the submonoid of $\PT(\Omega)$ consisting of all
full transformations on $\Omega$, by $\I(\Omega)$
the \textit{symmetric inverse monoid} on $\Omega$, i.e.
the inverse submonoid of $\PT(\Omega)$ consisting of all
partial permutations on $\Omega$,
and by $\Sym(\Omega)$ the \textit{symmetric group} on $\Omega$,
i.e. the subgroup of $\PT(\Omega)$ consisting of all
permutations on $\Omega$.
Recall that a semigroup $S$ is called inverse if, for each $s\in S$, there exists a unique $s'\in S$ with $s=ss's$ and $s'=s'ss'$ ($s'$ is
called inverse of $s$).
If $\Omega$ is a finite set with $n$ elements ($n\in\N$),
say $\Omega=\Omega_n=\{1,2,\ldots,n\}$, as usual, we denote
$\PT(\Omega)$, $\T(\Omega)$, $\I(\Omega)$ and $\Sym(\Omega)$ simply by $\PT_n$, $\T_n$, $\I_n$ and $\Sym_n$, respectively.
An element $\alpha$ belonging to $\PT(\Omega)$ with $\dom(\alpha)=\{a_{1},a_{2},\ldots,a_{k}\}, k\in\{1,2,\ldots,n\}$, can be written in the following form
$$\alpha=\begin{pmatrix}
a_{1}&a_{2}&\cdots&a_{k}\\
a_{1}\alpha&a_{2}\alpha&\cdots&a_{k}\alpha
\end{pmatrix}.$$

\begin{example}
\rm Let $\alpha, \beta \in \PT_{6}$ be the following partial transformations:
$$\alpha=\begin{pmatrix}
1&2&4&5&6\\
3&2&5&3&1
\end{pmatrix}
\quad
\textrm{and}
\quad
\beta=\begin{pmatrix}
2&3&4&6\\
2&1&6&4
\end{pmatrix}.$$
Then, for the compositions $\alpha\beta$ and $\beta\alpha$, we obtain
$$\alpha\beta=\begin{pmatrix}
1&2&5\\
1&2&1
\end{pmatrix}
\quad
\textrm{and}
\quad
\beta\alpha=\begin{pmatrix}
2&3&4&6\\
2&3&1&5
\end{pmatrix}.$$
\end{example}

\smallskip

Now, let $G=(V,E)$ be a finite simple connected graph. The (\textit{geodesic}) \textit{distance} between two vertices $x$ and $y$ of $G$, denoted by $\d_G(x,y)$, is the length of a shortest path between $x$ and $y$, i.e. the number of edges in a shortest path between $x$ and $y$.

Let $\alpha\in\PT(V)$. We say that $\alpha$ is a \textit{partial isometry} or \textit{distance preserving partial transformation} of $G$ if
$$
\d_G(x\alpha,y\alpha) = \d_G(x,y) ,
$$
for all $x,y\in\dom(\alpha)$. Denote by $\DP(G)$ the subset of $\PT(V)$ of all partial isometries of $G$.
Clearly, $\DP(G)$ is a submonoid of $\PT(V)$.
As a consequence of the property
$\d_G(x,y)=0$ if and only if $x=y$,
for all $x,y\in V$, it immediately follows that $\DP(G)\subseteq\I(V)$.
Moreover, $\DP(G)$ is an inverse submonoid of $\I(V)$
(see \cite{Fernandes&Paulista:2022arxiv}).

\smallskip

Observe that, if $G=(V,E)$ is a complete graph, i.e. $E=\{\{x,y\}\mid x,y\in V, x\neq y\}$, then $\DP(G)=\I(V)$.

\smallskip

For $n\in\N$, consider the undirected path $P_n$ with $n$ vertices, i.e.
$$
P_n=\left(\{1,2,\ldots,n\},\{\{i,i+1\}\mid i=1,2,\ldots,n-1\}\right).
$$
Then, obviously, $\DP(P_n)$ coincides with the monoid
$$
\DP_n=\{\alpha\in\I_n \mid |i\alpha-j\alpha|=|i-j|, \mbox{for all $i,j\in\dom(\alpha)$}\}
$$
of all partial isometries on $\Omega_n$.
The study of partial isometries on $\Omega_n$ was initiated
by Al-Kharousi et al.~\cite{AlKharousi&Kehinde&Umar:2014,AlKharousi&Kehinde&Umar:2016}.
The first of these two papers is dedicated to investigating some combinatorial properties of
the monoid $\DP_n$ and of its submonoid $\ODP_n$ of all order-preserving (considering the usual order of $\N$) partial isometries, in particular, their cardinalities. The second paper presents the study of some of their algebraic properties, namely Green's structure and ranks.
Presentations for both the monoids $\DP_n$ and $\ODP_n$ were given by Fernandes and Quinteiro in \cite{Fernandes&Quinteiro:2016}
and the maximal subsemigroups of $\ODP_n$ were characterized by Dimitrova in \cite{Dimitrova:2013}.

The monoid $\DP(S_n)$ of all partial isometries of a star graph $S_n$ with $n$ vertices ($n\geqslant1$)
was considered by Fernandes and Paulista in \cite{Fernandes&Paulista:2022arxiv}.
They determined the rank and size of $\DP(S_n)$ as well as described its Green's relations.
A presentation for $\DP(S_n)$ was also exhibited in \cite{Fernandes&Paulista:2022arxiv}.

Next, for $n\geqslant3$, consider the \textit{cycle graph}
$$
C_n=(\{1,2,\ldots, n\}, \{\{i,i+1\}\mid i=1,2,\ldots,n-1\}\cup\{\{1,n\}\})
$$
with $n$ vertices.  Notice that cycle graphs and cycle subgraphs play a fundamental role in Graph Theory.
The monoid $\DP(C_n)$ of all partial isometries of the cycle graph $C_n$ was studied by Fernandes and Paulista in \cite{Fernandes&Paulista:2022sub}.
They showed that $\DP(C_n)$ is an inverse submonoid of the monoid of all oriented partial
permutations on a chain with $n$ elements and, moreover, that it coincides with the inverse submonoid of $\I_n$
formed by all restrictions of a dihedral subgroup of $\Sym_n$ of order $2n$.
Therefore, in \cite{Fernandes&Paulista:2022sub}, $\DP(C_n)$ was called the \textit{dihedral inverse monoid} on $\Omega_n$ and, in this paper, from now on,
we denote $\DP(C_n)$ by the most appropriate notation $\DI_n$.
Recall also that in \cite{Fernandes&Paulista:2022sub} it was determined the cardinality and rank of $\DI_n$ as well as descriptions of its Green's relations and, furthermore,
presentations for $\DI_n$ were also given in that paper.

\smallskip

Next, suppose that $\Omega_n$ is a chain, e.g. $\Omega_n=\{1<2<\cdots<n\}$. A partial transformation $\alpha\in\PT_n$ is called \textit{order-preserving}
[\textit{order-reversing}] if, $x\leqslant y$ implies $x\alpha\leqslant y\alpha$
[$x\alpha\geqslant y\alpha$], for all $x,y \in \dom(\alpha)$.
A partial transformation is said to be \textit{monotone} if it is order-preserving or order-reversing.
It is clear that the product of two order-preserving or of two
order-reversing transformations is order-preserving and
the product of an order-preserving transformation by an
order-reversing transformation, or vice-versa, is
order-reversing.
We denote by $\PO_n$
the submonoid of $\PT_n$ of all order-preserving
transformations and by $\POD_n$ the submonoid of $\PT_n$
of all monotone transformations.
Let also $\POI_n=\PO_n\cap\I_n$,
the monoid of all order-preserving partial permutations of $\Omega_n$,
 and $\PODI_n=\POD_n\cap\I_n$,
 the monoid of all monotone partial permutations of $\Omega_n$,
 which are inverse submonoids of $\PT_n$.

Let $s=(a_1,a_2,\ldots,a_t)$
be a sequence of $t$ ($t\geqslant0$) elements
from the chain $\Omega_n$.
We say that $s$ is \textit{cyclic}
[\textit{anti-cyclic}] if there
exists no more than one index $i\in\{1,\ldots,t\}$ such that
$a_i>a_{i+1}$ [$a_i<a_{i+1}$],
where $a_{t+1}$ denotes $a_1$.
We also say that $s$ is \textit{oriented} if $s$ is cyclic or $s$ is anti-cyclic
(see \cite{Catarino&Higgins:1999,Higgins&Vernitski:2022,McAlister:1998}).
Given a partial transformation $\alpha\in\PT_n$ such that
$\dom(\alpha)=\{a_1<\cdots<a_t\}$, with $t\geqslant0$, we
say that $\alpha$ is \textit{orientation-preserving}
[\textit{orientation-reversing}, \textit{oriented}] if the sequence of its images
$(a_1\alpha,\ldots,a_t\alpha)$ is cyclic [anti-cyclic, oriented].
It is easy to show
that the product of two orientation-preserving or of two
orientation-reversing transformations is orientation-preserving and
the product of an orientation-preserving transformation by an
orientation-reversing transformation, or vice-versa, is
orientation-reversing.
We denote by $\POP_n$ the submonoid of $\PT_n$
of all orientation-preserving
transformations and by $\POR_n$ the
submonoid of $\PT_n$ of all
oriented  transformations.
Consider also the inverse submonoids $\POPI_n=\POP_n\cap\I_n$,
of all orientation-preserving partial permutations,
and $\PORI_n=\POR_n\cap\I_n$,
of all oriented partial permutations, of $\PT_n$.

Notice that for $n\geq 3$, $\POI_n\subsetneq\PODI_n\subsetneq\PORI_n$
and $\POI_n\subsetneq\POPI_n\subsetneq\PORI_n$, by definition.

\begin{example}
\rm Let us consider the following transformations of $\I_{5}$:
$$\alpha_{1}=\begin{pmatrix}
1&2&3\\
1&4&5
\end{pmatrix},
\quad
\alpha_{2}=\begin{pmatrix}
2&3&4&5\\
5&3&2&1
\end{pmatrix},
\quad
\alpha_{3}=\begin{pmatrix}
1&3&4&5\\
2&3&4&1
\end{pmatrix}
\quad
\textrm{and}
\quad
\alpha_{4}=\begin{pmatrix}
1&2&3&4&5\\
2&1&5&4&3
\end{pmatrix}.$$
Then, we have $\alpha_{1}\in \POI_{5}$, $\alpha_{2}\in \PODI_{5}\setminus \POI_{5}$, $\alpha_{3}\in \POPI_{5}\setminus \POI_{5}$ and $\alpha_{4}\in \PORI_{5}\setminus \POPI_{5}$.
\end{example}

\smallskip

Now, let us consider the following permutations of $\Omega_n$ of order $n$ and $2$, respectively:
$$
g=\begin{pmatrix}
1&2&\cdots&n-1&n\\
2&3&\cdots&n&1
\end{pmatrix}
\quad\text{and}\quad
h=\begin{pmatrix}
1&2&\cdots&n-1&n\\
n&n-1&\cdots&2&1
\end{pmatrix}.
$$

It is clear that $g,h\in\DI_n$.
Moreover, for $n\geqslant3$, $g$ together with $h$ generate the well-known \textit{dihedral group} $\D_{2n}$ of order $2n$
(considered as a subgroup of $\Sym_n$). In fact, for $n\geqslant3$,
$$
\D_{2n}=\langle g,h\mid g^n=1,h^2=1, hg=g^{n-1}h\rangle=\{\id,g,g^2,\ldots,g^{n-1}, h,hg,hg^2,\ldots,hg^{n-1}\},
$$
where $\id$ denotes the identity transformation on $\Omega_n$,
and we have
$$
g^k=\begin{pmatrix}
1&2&\cdots&n-k&n-k+1&\cdots&n\\
1+k&2+k&\cdots&n&1&\cdots&k
\end{pmatrix},
\quad\text{i.e.}\quad
ig^k=\left\{\begin{array}{ll}
i+k & \mbox{if $1\leqslant i\leqslant n-k$}\\
i+k-n & \mbox{if $n-k+1\leqslant i\leqslant n$,}
\end{array}\right.
$$
and
$$
hg^k=\begin{pmatrix}
1&\cdots&k&k+1&\cdots&n\\
k&\cdots&1&n&\cdots&k+1
\end{pmatrix},
\quad\text{i.e.}\quad
ihg^k=\left\{\begin{array}{ll}
k-i+1 & \mbox{if $1\leqslant i\leqslant k$}\\
n+k-i+1 & \mbox{if $k+1\leqslant i\leqslant n$,}
\end{array}\right.
$$
for $0\leqslant k\leqslant n-1$.
Denote also by $\C_n$ the \textit{cyclic group} of order $n$ generated by $g$, i.e.
$$
\C_n=\langle g\mid g^n=1\rangle=\{\id,g,g^2,\ldots,g^{n-1}\}.
$$

\smallskip

Until the end of this paper, we will consider $n\geqslant3$.

\smallskip

For any two vertices $x$ and $y$ of $C_n$, we now denote the distance $\d_{C_n}(x,y)$ simply by $\d(x,y)$.
Notice that, we have
$$
\d(x,y)=\min \{|x-y|,n-|x-y|\}
= \left\{ \begin{array}{ll}
 |x-y| &\mbox{if  $|x-y|\leqslant\frac{n}{2}$}\\
n-|x-y| &\mbox{if $|x-y|>\frac{n}{2}$}
\end{array} \right.
$$
and so $0\leqslant\d(x,y)\leqslant\frac{n}{2}$,
for all $x,y \in \{1,2,\ldots,n\}$.
Observe also that
$$
\d(x,y)=\frac{n}{2}
\quad\Leftrightarrow\quad
|x-y|=\frac{n}{2}
\quad\Leftrightarrow\quad
n-|x-y|=\displaystyle\frac{n}{2}
\quad\Leftrightarrow\quad
|x-y|=n-|x-y|,
$$
in which case $n$ is even.

\smallskip

Recall that $\DI_n$ is the submonoid of the monoid $\PORI_n$ whose elements are precisely all restrictions of the dihedral group $\D_{2n}$ of order $2n$.
Let $\alpha \in \PT_n$ and let $A \subset \dom(\alpha)$. We denote by $\alpha|_A$ the restriction of $\alpha$ to $A$.
Moreover, it is also known exactly how many extensions in $\D_{2n}$ each element of $\DI_n$ has:

\begin{lemma}[{\cite[Lemma 1.1]{Fernandes&Paulista:2022sub}}] \label{fundlemma}
Let $\alpha \in \PT_n$. Then $\alpha \in\DI_n$ if and only if there exists $\sigma \in \D_{2n}$
such that $\alpha=\sigma|_{\dom(\alpha)}$.
Furthermore, for $\alpha \in \DI_n$, one has:
\begin{enumerate}
\item If either $|\dom(\alpha)|= 1$ or $|\dom(\alpha)|= 2$ and $\d(\min \dom(\alpha),\max \dom(\alpha))=\frac{n}{2}$
(in which case $n$ is even),
then there exists exactly two (distinct) permutations $\sigma,\sigma' \in\D_{2n}$ such that $\alpha= \sigma|_{\dom(\alpha)} = \sigma'|_{\dom(\alpha)}$;

\item If either $|\dom(\alpha)|= 2$ and $\d(\min \dom(\alpha),\max \dom(\alpha)) \neq \frac{n}{2}$ or $|\dom(\alpha)|\geqslant 3$,
then there exists exactly one permutation $\sigma \in\mathcal{D}_{2n}$ such that $\alpha= \sigma|_{\dom(\alpha)}$.
\end{enumerate}
\end{lemma}

Notice that for an even $n$, we have
$$
\begin{array}{rcl}
\B_2 & = & \{\alpha\in\DI_n\mid |\mbox{$\dom(\alpha)|=2$ and $\d(\min \dom(\alpha),\max \dom(\alpha))=\frac{n}{2}$}\} \\
& = &
\left\{
\begin{pmatrix}
i&i+\frac{n}{2}\\
j&j+\frac{n}{2}
\end{pmatrix},
\begin{pmatrix}
i&i+\frac{n}{2}\\
j+\frac{n}{2}&j
\end{pmatrix}
\mid
1\leqslant i,j\leqslant \frac{n}{2}
\right\}
\end{array}
$$
and so $|\B_2|=2(\frac{n}{2})^2=\frac{1}{2}n^2$.

\medskip

In this paper, we study three submonoids of $\DI_n$, namely
$\OPDI_n=\DI_n\cap\POPI_n$, the monoid of all orientation-preserving partial isometries of $C_n$,
$\MDI_n=\DI_n\cap\PODI_n$, the monoid of all monotone partial isometries of $C_n$,
and $\ODI_n=\DI_n\cap\POI_n$, the monoid of all order-preserving partial isometries of $C_n$.
Observe that $\DI_n$, $\OPDI_n$, $\MDI_n$ and $\ODI_n$ are all inverse submonoids of the symmetric inverse monoid $\I_n$,
$\ODI_n\subseteq\MDI_n$ and $\ODI_n\subseteq\OPDI_n$.
Also, observe that $\OPDI_3=\POPI_3$, $\MDI_3=\PODI_3$ and $\ODI_3=\POI_3$.

\begin{example}
\rm Let us consider the following transformations of $\DI_{5}$:
$$\alpha_{1}=\begin{pmatrix}
2&4&5\\
1&3&4
\end{pmatrix},
\quad
\alpha_{2}=\begin{pmatrix}
1&2&3\\
3&2&1
\end{pmatrix}
\quad
\textrm{and}
\quad
\alpha_{3}=\begin{pmatrix}
1&2&3&5\\
3&4&5&2
\end{pmatrix}.
$$
Then, we have $\alpha_{1}\in \ODI_{5}$, $\alpha_{2}\in \MDI_{5}\setminus \ODI_{5}$ and $\alpha_{3}\in \OPDI_{5}\setminus \ODI_{5}$.
\end{example}

\smallskip

This paper investigates algebraic, combinatorial and rank properties of each of the monoids $\ODI_{n}$, $\MDI_{n}$ and $\OPDI_{n}$.
In particular, we determine the cardinality (Section \ref{cards}, Theorem \ref{sizeopcoc}), describe the Green's relation $\mathcal{J}$ (Section \ref{greens}, Theorem \ref{greenJ}) and calculate the rank of each of these monoids.
The main results of the paper are presented in Section \ref{ranks} which is dedicated to establish generating sets (Proposition \ref{gensets}) and to determine the ranks of these three monoids (Theorem \ref{rankth}).

Recall that, for a monoid $M$, the Green's relations $\mathcal{L}$, $\mathcal{R}$, $\mathcal{J}$ and $\mathcal{H}$
are defined by
\begin{itemize}
\item $a\mathcal{L}b$ if and only if $Ma=Mb$ for $a,b\in M$,

\item $a\mathcal{R}b$ if and only if $aM=bM$ for $a,b\in M$,

\item $a\mathcal{J}b$ if and only if $MaM=MbM$ for $a,b\in M$, and

\item $\mathcal{H}=\mathcal{L}\cap \mathcal{R}$.
\end{itemize}

Green's relations are very useful tool in the study of semigroups/monoids. They help us to gain a deeper understanding of the internal structure of semigroups: we can identify subsemigroups, study the idempotent elements, and explore the congruence properties within the semigroup. Overall, the importance of Green's relations in semigroup theory lies in their ability to provide a systematic way to study and classify elements within semigroup, leading to valuable perceptions into their algebraic properties.

The notion rank or dimension belongs primarily to linear algebra. In semigroups, we normally define the \textit{rank} of a semigroup $S$ as
being the minimum size of a generating set of $S$, i.e. the minimum of the set $\{|X|\mid \mbox{$X\subseteq S$ and $X$ generates $S$}\}$.
For a discussion on rank properties in finite semigroups and other possible definitions, see \cite{Howie&Ribeiro:1999}.
The rank provides information about the complexity and algebraic properties of the semigroup. It helps us to understand the diversity of elements within the semigroup as higher-rank semigroups often exhibit more intricate behavior and possess a richer variety of elements.

\smallskip

For $n\geqslant3$,
it is well-known that $\Sym_n$
has rank $2$ (as a semigroup, a monoid or a group) and
$\T_n$, $\I_n$ and $\PT_n$ have
ranks $3$, $3$ and $4$, respectively.
The survey \cite{Fernandes:2002survey} presents
these results and similar ones for other classes of transformation monoids,
in particular, for monoids of order-preserving transformations and
for some of their extensions.
For example, the rank of the extensively studied monoid of all order-preserving transformations of a chain with $n$ elements is $n$,
a result proved by Gomes and Howie \cite{Gomes&Howie:1992} in 1992.
More recently, for instance, the papers
\cite{
Araujo&al:2015,
Dimitrova&al:2020,
Dimitrova&al:2021,
Dimitrova&Koppitz:2017,
Dimitrova&al:2017,
Fernandes&al:2014,
Fernandes&al:2019,
Fernandes&Quinteiro:2014,
Fernandes&Sanwong:2014}
are dedicated to the computation of the ranks of certain classes of transformation semigroups or monoids.

\medskip

For general background on Semigroup Theory and standard notations, we refer to Howie's book \cite{Howie:1995}.

\smallskip

We would like to point out that we made considerable use of computational tools, namely GAP \cite{GAP4}.

\section{Cardinality}\label{cards}

We begin this paper with some combinatorial considerations. Enumerative problems of an essentially combinatorial nature arise naturally
in the study of semigroups of transformations. Our main aim in this section is to find a formula for $|\ODI_n|$, $|\MDI_n|$ and $|\OPDI_n|$, respectively.

\medskip

By applying Lemma \ref{fundlemma} and counting all possible distinct orientation-preserving and order-preserving restrictions of permutations from $\D_{2n}$, we have:

\begin{theorem}\label{sizeopcoc}
One has
$$
|\ODI_n|= 3\cdot2^n+\frac{(n+1)n(n-1)}{6}-\frac{1+(-1)^n}{8}n^2-2n-2
$$
and
$$
|\OPDI_n|=
n2^n + \frac{n^2(n-1)}{2} -\frac{1+(-1)^n}{4}n^2-n+1.
$$
\end{theorem}
\begin{proof}
Let $\A=\{\alpha\in\DI_n\mid |\dom(\alpha)|\leqslant1\}$. Clearly,
$\A=\{\alpha\in\OPDI_n\mid |\dom(\alpha)|\leqslant1\}=\{\alpha\in\ODI_n\mid |\dom(\alpha)|\leqslant1\}$.
It is also clear that $|\A|=1+n^2$. Therefore,
in view of Lemma \ref{fundlemma}, to determine the sizes of $\ODI_n$ and $\OPDI_n$, it suffices to count how many distinct restrictions of permutations of $\D_{2n}$ with rank greater than or equal to $2$ are order-preserving and orientation-preserving, respectively.

\smallskip

First, we determine the cardinality of the set $\B=\{\alpha\in\ODI_n\mid |\dom(\alpha)|\geqslant2\}$. Let $k\in\{0,1,\ldots,n-1\}$.
Clearly, the only order-preserving restrictions of $hg^k$, with rank greater than or equal to $2$,
are of the form $hg^k|_{\{i<j\}}$, with $1\leqslant i\leqslant k$ and $k+1\leqslant j\leqslant n$.
Hence, we have $k\times(n-k)$ order-preserving restrictions of $hg^k$ with rank greater than or equal to $2$.

On the other hand, any order-preserving restriction of $g^k$ has its domain contained in $\{1,\ldots,n-k\}$ or in $\{n-k+1,\ldots,n\}$,
whence $g^k$ has $\sum_{i=2}^{n-k}\binom{n-k}{i}+\sum_{i=2}^{k}\binom{k}{i}$ order-preserving restrictions with rank greater than or equal to $2$.

Observe that, if $n$ is even then
$$
\B_2\cap\B=\left\{
\begin{pmatrix}
i&i+\frac{n}{2}\\
j&j+\frac{n}{2}
\end{pmatrix}
\mid
1\leqslant i,j\leqslant \frac{n}{2}
\right\},
$$
whence we have $|\B_2\cap\B|=(\frac{n}{2})^2=\frac{1}{4}n^2$ elements in $\B$ with exactly two extensions in $\D_{2n}$,
while the remaining elements only have one. Conversely, for an odd $n$, all elements of $\B$ have exactly one extension in $\D_{2n}$.
Thus
$$
|\B|= \left\{\begin{array}{ll} \displaystyle
\sum_{k=0}^{n-1}\left(k\times(n-k)\right) + \sum_{k=0}^{n-1}\left(\sum_{i=2}^{n-k}\binom{n-k}{i}+\sum_{i=2}^{k}\binom{k}{i}\right)
& \mbox{if $n$ is odd}
\\\\ \displaystyle
\sum_{k=0}^{n-1}\left(k\times(n-k)\right) + \sum_{k=0}^{n-1}\left(\sum_{i=2}^{n-k}\binom{n-k}{i}+\sum_{i=2}^{k}\binom{k}{i}\right) - \frac{1}{4}n^2
& \mbox{if $n$ is even.}
\end{array}\right.
$$

Now, since
\begin{align*}
\sum_{k=0}^{n-1}\left(k\times(n-k)\right) = \sum_{k=1}^{n-1}\left(k\times(n-k)\right) =
n\sum_{k=1}^{n-1}k - \sum_{k=1}^{n-1}k^2 = \qquad\qquad\qquad\qquad\\
n\frac{1+(n-1)}{2}(n-1) - \frac{(n-1)n(2(n-1)+1)}{6} = \frac{(n+1)n(n-1)}{6}
\end{align*}
and
\begin{align*}
\sum_{k=0}^{n-1}\left(\sum_{i=2}^{n-k}\binom{n-k}{i}+\sum_{i=2}^{k}\binom{k}{i}\right) =
\sum_{k=0}^{n-1}\left( (2^{n-k}-n+k-1)+(2^k-k-1) \right) =
\sum_{k=1}^{n}2^k + \sum_{k=0}^{n-1}2^k -\sum_{k=0}^{n-1}(n+2) =\\
(2^{n+1}-1-1)+(2^n-1)-n(n+2) =
3\cdot2^n-n^2-2n-3,
\end{align*}
the result about $|\ODI_n|=|\A|+|\B|$ immediately follows.

\smallskip

Next, we determine the cardinality of the set $\C=\{\alpha\in\OPDI_n\mid |\dom(\alpha)|\geqslant2\}$. Let $k\in\{0,1,\ldots,n-1\}$.
The orientation-preserving restrictions of $hg^k$, with rank greater than or equal to $2$, are all its order-preserving restrictions (which as seen above must have rank $2$) together with all its order-reversing restrictions of rank $2$. Hence,
we have $k\times(n-k) + \binom{k}{2} + \binom{n-k}{2}$
orientation-preserving restrictions of $hg^k$ with rank greater than or equal to $2$. Since all restrictions of $g^k$ are orientation-preserving
and, for an even $n$, $\B_2\subseteq\C$ with $|\B_2|=\frac{1}{2}n^2$, we have
$$
|\C|= \left\{\begin{array}{ll} \displaystyle
\sum_{k=0}^{n-1}\left(k\times(n-k)+  \binom{k}{2} + \binom{n-k}{2}\right) + \sum_{k=0}^{n-1}\sum_{i=2}^{n}\binom{n}{i} & \mbox{if $n$ is odd}
\\\\ \displaystyle
\sum_{k=0}^{n-1}\left(k\times(n-k)+  \binom{k}{2} + \binom{n-k}{2}\right) + \sum_{k=0}^{n-1}\sum_{i=2}^{n}\binom{n}{i} - \frac{1}{2}n^2
& \mbox{if $n$ is even.}
\end{array}\right.
$$
Now, from
\begin{align*}
\sum_{k=0}^{n-1}\left(k\times(n-k)+  \binom{k}{2} + \binom{n-k}{2}\right) =
\frac{(n+1)n(n-1)}{6} + 2\sum_{k=0}^{n-1}\binom{k}{2} + \binom{n}{2} = \qquad\qquad\quad\\
\frac{(n+1)n(n-1)}{6} + 2\binom{n}{3} + \binom{n}{2} = \frac{n^2(n-1)}{2}
\end{align*}
and
$$
\sum_{k=0}^{n-1}\sum_{i=2}^{n}\binom{n}{i} = \sum_{k=0}^{n-1}\left(2^n-n-1\right)=n(2^n-n-1) = n2^n-n^2-n,
$$
the result about $|\OPDI_n|=|\A|+|\C|$ also follows.
\end{proof}

Next, we exemplify the previous proof with the calculation of the cardinality of $\ODI_{4}$.

\begin{example}
\rm Let $\A=\{\alpha\in\ODI_4\mid |\dom(\alpha)|\leqslant1\}$.
It is clear that $\emptyset \in \A$ and if $\alpha \in \A$ with $|\dom(\alpha)|=1$ then $\alpha \in \left\{
\trans{i \\ j}
 \mid 1 \leqslant i, j \leqslant 4\right\}$.
So, we have $|\A|=1+4^2=17$.

Therefore, in view of Lemma \ref{fundlemma}, to determine $|\ODI_4|$, it suffices to count how many distinct restrictions of permutations of
$\D_{2\cdot4}=\{\id,g,g^{2},g^{3},h,hg,hg^{2},hg^{3}\}$, with rank greater than or equal to $2$, are order-preserving.

\smallskip

Let $\B=\{\alpha\in\ODI_4 \mid |\dom(\alpha)|\geqslant2\}$.
Clearly, for $k\in\{0,1,2,3\}$, the only order-preserving restrictions of $hg^k$, with rank greater than or equal to $2$,
are of the form $hg^k|_{\{i<j\}}$, with $1\leqslant i\leqslant k$ and $k+1\leqslant j\leqslant 4$.
Recall that
$g=\trans{
1&2&3&4\\
2&3&4&1
}$ and $h=\trans{
1&2&3&4\\
4&3&2&1
}$.
Let $\alpha = hg^k|_{\{i<j\}} \in \B$. Then
\begin{multline*}
\alpha \in \left\{
\trans{
1&2\\
1&4
}=hg|_{\{1,2\}},
\trans{
1&3\\
1&3
}=hg|_{\{1,3\}},
\trans{
1&4\\
1&2
}=hg|_{\{1,4\}},
\trans{
1&3\\
2&4
}=hg^{2}|_{\{1,3\}},
\trans{
1&4\\
2&3
}=hg^{2}|_{\{1,4\}},  \right. \\ \left.
\trans{
2&3\\
1&4
}=hg^{2}|_{\{2,3\}},
\trans{
2&4\\
1&3
}=hg^{2}|_{\{2,4\}},
\trans{
1&4\\
3&4
}=hg^{3}|_{\{1,4\}},
\trans{
2&4\\
2&4
}=hg^{3}|_{\{2,4\}},
\trans{
3&4\\
1&4
}=hg^{3}|_{\{3,4\}}\right\}.
\end{multline*}
Hence, we have $k\times(4-k)$ order-preserving restrictions of $hg^k$ with rank greater than or equal to $2$. Altogether, we have $(1\times3)+(2\times2)+(3\times1) = 10$ such order-preserving transformations.

On the other hand,  for $k\in\{0,1,2,3\}$, any order-preserving restriction of $g^k$ has its domain contained in $\{1,\ldots,4-k\}$ or in $\{5-k,\ldots,4\}$.
Therefore,
for $\alpha = g^k|_{\dom(\alpha)} \in \B$ with $\dom(\alpha)\subseteq \{1,\ldots,4-k\}$ or $\dom(\alpha) \subseteq \{5-k,\ldots,4\}$, we have
\begin{multline*}
\alpha \in \left\{
\trans{
1&2&3&4\\
1&2&3&4
}=\id|_{\{1,2,3,4\}},
\trans{
1&2&3\\
1&2&3
}=\id|_{\{1,2,3\}},
\trans{
1&2&4\\
1&2&4
}=\id|_{\{1,2,4\}},
\trans{
1&3&4\\
1&3&4
}=\id|_{\{1,3,4\}},  \right. \\ \left.
\trans{
2&3&4\\
2&3&4
}=\id|_{\{2,3,4\}},
\trans{
1&2\\
1&2
}=\id|_{\{1,2\}},
\trans{
1&3\\
1&3
}=\id|_{\{1,3\}},
\trans{
1&4\\
1&4
}=\id|_{\{1,4\}},
\trans{
2&3\\
2&3
}=\id|_{\{2,3\}},
\trans{
2&4\\
2&4
}=\id|_{\{2,4\}},  \right. \\ \left.
\trans{
3&4\\
3&4
}=\id|_{\{3,4\}},
\trans{
1&2&3\\
2&3&4
}=g|_{\{1,2,3\}},
\trans{
1&2\\
2&3
}=g|_{\{1,2\}},
\trans{
1&3\\
2&4
}=g|_{\{1,3\}},
\trans{
2&3\\
3&4
}=g|_{\{2,3\}},
\trans{
1&2\\
3&4
}=g^{2}|_{\{1,2\}},  \right. \\ \left.
\trans{
3&4\\
1&2
}=g^{2}|_{\{3,4\}},
\trans{
2&3&4\\
1&2&3
}=g^{3}|_{\{2,3,4\}},
\trans{
2&3\\
1&2
}=g^{3}|_{\{2,3\}},
\trans{
2&4\\
1&3
}=g^{3}|_{\{2,4\}},
\trans{
3&4\\
2&3
}=g^{3}|_{\{3,4\}}\right\}.
\end{multline*}
Hence, $g^k$ has $\sum_{i=2}^{4-k}\binom{4-k}{i}+\sum_{i=2}^{k}\binom{k}{i} = 16+5=21$ order-preserving restrictions with rank greater than or equal to $2$.

Observe that, if $\alpha \in \B_2=\{\alpha\in\ODI_4\mid |\dom(\alpha)|=2 \textrm{ and } \d(\min \dom(\alpha),\max \dom(\alpha))=2\} \subseteq \B$ then
\begin{multline*}
\alpha \in \left\{
\trans{
1&3\\
1&3
}=\id|_{\{1,3\}}=hg|_{\{1,3\}},
\trans{
1&3\\
2&4
}=g|_{\{1,3\}}=hg^{2}|_{\{1,3\}},
\trans{
2&4\\
2&4
}=\id|_{\{2,4\}}=hg^{3}|_{\{2,4\}}, \right. \\ \left.
\trans{
2&4\\
1&3
}=g^{3}|_{\{2,4\}}=hg^{2}|_{\{2,4\}}
\right\},
\end{multline*}
whence we have $|\B_2\cap\B|=4$ elements in $\B$ with exactly two extensions in $\D_{2\cdot4}$.
The remaining elements only have one extension in $\D_{2\cdot4}$. Therefore, we have $|\B|=10+21-4=27$.

Thus, we obtain $|\ODI_{4}|=|\A|+|\B|=17+27=44$.
\end{example}

\smallskip

The previous approach could also be applied to count the elements of $\MDI_n$.
However, since all $n^2+1$ elements of $\MDI_n$ with rank less than or equal to $1$ are order-preserving and  the mapping
$$
\begin{array}{ccc}
\{\alpha\in\ODI_n\mid |\im(\alpha)|\geqslant2\}&\longrightarrow&\{\alpha\in\MDI_n\setminus\ODI_n\mid |\im(\alpha)|\geqslant2\}\\
\alpha & \longmapsto & \alpha h
\end{array}
$$
is a bijection (notice $\alpha=\alpha h^2$, for all $\alpha\in\PT_n$), then $|\MDI_n|=2|\ODI_n|-n^2-1$.
Hence, as an immediate consequence of Theorem \ref{sizeopcoc}, we have the following result.

\begin{theorem}\label{sizemc}
One has
$$
|\MDI_n|=
3\cdot2^{n+1}+\frac{(n+1)n(n-1)}{3} -\frac{5+(-1)^n}{4}n^2 -4n-5.
$$
\end{theorem}

\section{Green's relation $\mathscr{J}$}\label{greens}

The main result of this section is the description of the $\mathcal{J}$-relation for each of the monoids $\ODI_{n}$, $\MDI_{n}$ and $\OPDI_{n}$.

Given an inverse submonoid $M$ of $\I_n$,
it is well known that Green's relations $\mathscr{L}$, $\mathscr{R}$ and $\mathscr{H}$ of $M$ can be described as following:
for $\alpha, \beta \in M$,
\begin{itemize}
\item $\alpha \mathscr{L} \beta$ if and only if $\im(\alpha) = \im(\beta)$,

\item $\alpha \mathscr{R} \beta$ if and only if $\dom(\alpha) = \dom(\beta)$, and

\item $\alpha \mathscr{H} \beta$ if and only if $\im(\alpha) = \im(\beta)$ and $\dom(\alpha) = \dom(\beta)$.
\end{itemize}
In $\I_n$ we also have
\begin{itemize}
\item $\alpha \mathscr{J} \beta$ if and only if $|\dom(\alpha)| = |\dom(\beta)|$ (if and only if $|\im(\alpha)| = |\im(\beta)|$).
\end{itemize}

Observe that for a finite monoid, we always have $\mathscr{J} = \mathscr{D} (= \mathscr{L}\circ\mathscr{R} = \mathscr{R}\circ\mathscr{L})$.

Since the monoids $\ODI_n$, $\MDI_n$, and $\OPDI_n$ are inverse submonoids of $\I_n$, our main objective in this section is to give a description of Green's relation $\mathscr{J}$ for these monoids.\smallskip

To make the discussion of Green's relation $\mathscr{J}$ clearer and easier to follow, we divided it into lemmas. Lemma \ref{Grpr0} is a
 characterization of $\DI_n$, presented in \cite[Proposition 4.1.12, pages 67-81]{Paulista:2022}. It will be useful for the proof of Lemma \ref{relJ}.
Then we apply Lemma \ref{relJ} in the proof of Theorem \ref{greenJ}.

We have provided a simplified and concise proof for Lemma \ref{Grpr0}, as the original proof was complex and lengthy.

\begin{lemma}\label{Grpr0}
Let $\alpha \in \PORI_n$ be such that $\dom(\alpha) = \{i_1<i_2<\cdots<i_k\}$ with $k \in \{2,3,\ldots,n\}$.
Then $\alpha \in \DI_n$ if and only if $\d(i_1,i_k) = \d(i_1\alpha,i_k\alpha)$
and $\d(i_p,i_{p+1}) = \d(i_p\alpha,i_{p+1}\alpha)$ for $p=1,2,\ldots,k-1$.
\end{lemma}
\begin{proof}
If $\alpha \in \DI_n$ then, by definition, we have $\d(i_1,i_k) = \d(i_1\alpha,i_k\alpha)$
and $\d(i_p,i_{p+1}) = \d(i_p\alpha,i_{p+1}\alpha)$ for $p=1,2,\ldots,k-1$.

\smallskip

Conversely, suppose that  $\d(i_1,i_k) = \d(i_1\alpha,i_k\alpha)$
and $\d(i_p,i_{p+1}) = \d(i_p\alpha,i_{p+1}\alpha)$ for $p=1,2,\ldots,k-1$. Clearly, if $k=2$ or $k=3$, the result is trivial. So, we may assume that $k\geqslant3$.

If $\alpha\in\POPI_n$ then, by \cite[Proposition 3.1]{Fernandes:2000},
there exists $i\in\{0,1,\ldots,n-1\}$ and $\beta\in\POI_n$ such that $\alpha=g^i\beta$.
On the other hand, if $\alpha\not\in\POPI_n$ then $h\alpha\in\POPI_n$, whence
there exists also $i\in\{0,1,\ldots,n-1\}$ and $\beta\in\POI_n$ such that $h\alpha=g^i\beta$ and so $\alpha=hg^i\beta$.
Thus, in either case, there exist $i\in\{0,1,\ldots,n-1\}$, $j\in\{0,1\}$ and $\beta\in\POI_n$ such that $\alpha=h^jg^i\beta$.
Observe that we also have $\beta=g^{n-i}h^j\alpha$.

Suppose that $\dom(\beta) = \{i'_1<i'_2<\cdots<i'_k\}$ and let $t\in\{0,1,\ldots,k\}$ be such that $i'_t\leqslant i$ and $i'_{t+1}\geqslant i+1$
(with the obvious meaning for $t=0$ and $t=k$). Then
$$
(i_1,i_2,\ldots,i_k)=\left\{
\begin{array}{ll}
(i'_{t+1}g^{n-i},\ldots,i'_kg^{n-i},i'_1g^{n-i},\ldots,i'_tg^{n-i}) & \mbox{if $j=0$}\\ \\
(i'_tg^{n-i}h,\ldots,i'_1g^{n-i}h,i'_kg^{n-i}h,\ldots,i'_{t+1}g^{n-i}h) & \mbox{if $j=1$},
\end{array}
\right.
$$
from which it is a routine matter to show that $\d(i'_1,i'_k) = \d(i'_1\beta,i'_k\beta)$
and $\d(i'_p,i'_{p+1}) = \d(i'_p\beta,i'_{p+1}\beta)$ for $p=1,2,\ldots,k-1$, since $g,h\in\DI_n$.

Therefore, we may reduce our proof to order-preserving transformations and may assume that $\alpha\in\POI_n$.
Let $j_p=i_p\alpha$ for $p=1,2,\ldots,k$. Then $j_1<j_2<\cdots<j_k$.

\smallskip

First, we show that $i_{p+1}-i_p=j_{p+1}-j_p$ for $p=1,2,\ldots,k-1$.
Observe that $\sum_{p=1}^{k-1}(i_{p+1}-i_p)=i_k-i_1<n$ and so there exists at most one index
$r\in\{1,2,\ldots,k-1\}$ such that $i_{r+1}-i_r\geqslant\frac{n}{2}$. Similarly, there exists at most one index
$s\in\{1,2,\ldots,k-1\}$ such that $j_{s+1}-j_s\geqslant\frac{n}{2}$.
Also notice that for all $p\in\{1,2,\ldots,k-1\}$,
$i_{p+1}-i_p=\frac{n}{2}$ or $j_{p+1}-j_p=\frac{n}{2}$ implies that $i_{p+1}-i_p=\frac{n}{2}=j_{p+1}-j_p$.

In order to obtain a contradiction, suppose there exists $\ell\in\{1,2,\ldots,k-1\}$ such that $i_{\ell+1}-i_\ell\neq j_{\ell+1}-j_\ell$.
Let $r$ be the smallest of such indices. Since $\d(i_r,i_{r+1})=\d(j_r,j_{r+1})$, we get
$$
\d(i_r,i_{r+1})=i_{r+1}-i_r=n-j_{r+1}+j_r\quad\text{or}\quad \d(i_r,i_{r+1})=j_{r+1}-j_r=n-i_{r+1}+i_r.
$$
By considering $\alpha^{-1}$ instead of $\alpha$, we may assume, without loss of generality that  $\d(i_r,i_{r+1})=j_{r+1}-j_r$.
Hence $i_{r+1}-i_r>\frac{n}{2}$ and $j_{r+1}-j_r<\frac{n}{2}$.
Moreover, $r$ is the only index in $\{1,2,\ldots,k-1\}$ such that $i_{r+1}-i_r\geqslant\frac{n}{2}$.

\smallskip

We begin by assuming that $j_{p+1}-j_p<\frac{n}{2}$ for all $p=1,2,\ldots,k-1$.
Since $i_{p+1}-i_p<\frac{n}{2}$ for all $p\in\{1,2,\ldots,k-1\}\setminus\{r\}$, then
\begin{align*}
& \d(i_p,i_{p+1})=\d(j_p,j_{p+1}),~\mbox{for $p=1,2,\ldots,k-1$}
\quad \Longrightarrow \quad \sum_{p=1}^{k-1}\d(i_p,i_{p+1})=\sum_{p=1}^{k-1}\d(j_p,j_{p+1}) \\
 \Longrightarrow\quad &
 \sum_{p=1}^{r-1}(i_{p+1}-i_p) + (n-i_{r+1}+i_r) +  \sum_{p=r+1}^{k-1}(i_{p+1}-i_p) = \sum_{p=1}^{k-1}(j_{p+1}-j_p) \\
 \Longrightarrow\quad &
 (i_r-i_1) + (n-i_{r+1}+i_r) + (i_k-i_{r+1}) = j_k-j_1 \\
  \Longrightarrow\quad &
 (i_k-i_1) + (n-i_{r+1}+i_r) + (i_r-i_{r+1}) = j_k-j_1.
\end{align*}
On the other hand, as $\d(j_1,j_k)=\d(i_1,i_k)$ then $j_k-j_1=i_k-i_1$ or $j_k-j_1=n-i_k+i_1$.
If $j_k-j_1=i_k-i_1$ then $n-i_{r+1}+i_r=i_{r+1}-i_r>\frac{n}{2}$, which is a contradiction.
Thus $j_k-j_1=n-i_k+i_1$, whence $2(i_k-i_1+i_r-i_{r+1})=0$ and so $i_k-i_1=i_{r+1}-i_r$, which is again a contradiction (since $k\geqslant3$).

\smallskip

Therefore, there exists $s\in\{1,2,\ldots,k-1\}$ such that $j_{s+1}-j_s\geqslant\frac{n}{2}$, which is the only index under these conditions.
Moreover, $j_{s+1}-j_s>\frac{n}{2}$ and $s>r$.
Then, we have
\begin{align*}
& \d(i_p,i_{p+1})=\d(j_p,j_{p+1}),~\mbox{for $p=1,2,\ldots,k-1$}
\quad \Longrightarrow \quad \sum_{p=1}^{k-1}\d(i_p,i_{p+1})=\sum_{p=1}^{k-1}\d(j_p,j_{p+1}) \\
 \Longrightarrow\quad &
 \sum_{p=1}^{r-1}(i_{p+1}-i_p) + (n-i_{r+1}+i_r) +  \sum_{p=r+1}^{k-1}(i_{p+1}-i_p) =
\sum_{p=1}^{s-1}(j_{p+1}-j_p) + (n-j_{s+1}+j_s) +  \sum_{p=s+1}^{k-1}(j_{p+1}-j_p)   \\
 \Longrightarrow\quad &
 (i_r-i_1) + (n-i_{r+1}+i_r) + (i_k-i_{r+1}) = (j_s-j_1) + (n-j_{s+1}+j_s) + (j_k-j_{s+1}) \\
  \Longrightarrow\quad &
 (n+i_k-i_1) + 2(i_r-i_{r+1}) = (n+j_k-j_1) + 2(j_s-j_{s+1}) .
\end{align*}
Next, as $i_k-i_1\geqslant i_{r+1}-i_r>\frac{n}{2}$ and $j_k-j_1\geqslant j_{s+1}-j_s>\frac{n}{2}$, we have
$$
n-i_k+i_1=\d(i_1,i_k)=\d(j_1,j_k)=n-j_k+j_1
$$
and so $i_{r+1}-i_r=j_{s+1}-j_s$.
On the other hand, since $i_{s+1}-i_s<\frac{n}{2}$ and $j_{s+1}-j_s>\frac{n}{2}$, we have
$$
i_{s+1}-i_s=\d(i_s,i_{s+1})=\d(j_s,j_{s+1})=n-j_{s+1}+j_s,
$$
whence $i_{r+1}-i_r=n-i_{s+1}+i_s$ and so
$
n-1\geqslant i_{s+1}-i_r=n+i_s-i_{r+1}\geqslant n,
$
which is once again a contradiction.

\smallskip

Thus, we proved that $i_{p+1}-i_p=j_{p+1}-j_p$, for all $p\in\{1,2,\ldots,k-1\}$.

\smallskip

Now, let $1\leqslant p<q\leqslant k$. Then, we have
$
i_q-i_p=\sum_{t=p}^{q-1}(i_{t+1}-i_t)=\sum_{t=p}^{q-1}(j_{t+1}-j_t)=j_q-j_p,
$
from which follows also that $n-i_q+i_p=n-j_q+j_p$. Hence
$$
\d(i_p,i_q)=\left\{
\begin{array}{ll}
i_q-i_p &\mbox{if $i_q-i_p\leqslant\frac{n}{2}$}\\
n-i_q+i_p &\mbox{if $i_q-i_p>\frac{n}{2}$}
\end{array}
\right.
=\left\{
\begin{array}{ll}
j_q-j_p &\mbox{if $j_q-j_p\leqslant\frac{n}{2}$}\\
n-j_q+j_p &\mbox{if $j_q-j_p>\frac{n}{2}$}
\end{array}
\right.
=\d(j_p,j_q).
$$
Thus $\alpha\in\DI_n$, as required.
\end{proof}

Recall that $\id$ denotes the identity transformation on $\Omega_n$.
For $X\subseteq\Omega_n$, we denote by $\id_X$ the \emph{partial identity} with domain $X$, i.e. the restriction $\id|_X$ of the transformation $\id$ to the set $X$.

Now, for $A = \{i_1< i_2< \cdots< i_k\}\subseteq\Omega_n$ with $2\leqslant k \leqslant n$, define
$$
\d(A) = (d_1, d_2, \ldots, d_k),
$$
with $d_p = \d(i_p,i_{p+1})$, for $p = 1,\ldots,k-1$,
and $d_k = \d(i_1,i_k)$.
Take also $B = \{j_1< j_2< \cdots< j_k\}\subseteq\Omega_n$
and define $\delta_{A,B}$ as the only order-preserving transformation from $A$ onto $B$,
i.e.
$$
\delta_{A,B} = \left(
             \begin{array}{cccc}
               i_1 & i_2 & \cdots & i_k \\
               j_1 & j_2 & \cdots & j_k \\
             \end{array}
\right).
$$
Then, we have:

\begin{lemma}\label{relJ}
Let  $A = \{i_1< i_2< \cdots< i_k\}\subseteq\Omega_n$ and  $B = \{j_1< j_2< \cdots< j_k\}\subseteq\Omega_n$ with $2\leqslant k \leqslant n$. Then:
\begin{enumerate}
\item $\d(A)=\d(B)$ if and only if there exists an order-preserving partial isometry from $A$ onto $B$ (i.e. if and only if $\delta_{A,B}\in\ODI_n$);
\item $\d(A)=\d(Bh)$ if and only if there exists an order-reversing partial isometry from $A$ onto $B$;
\item $\d(A)=\d(Bg^{-s})$ for some $0\leqslant s\leqslant n-1$ if and only if there exists an orientation-preserving partial isometry from $A$ onto $B$.
\end{enumerate}
\end{lemma}
\begin{proof}
To prove 1, first suppose that $\d(A)=\d(B)$. Then, we have, for $1\leqslant p\leqslant k-1$,
$\d(i_p,i_{p+1}) = \d(j_p,j_{p+1})=\d(i_p\delta_{A,B},i_{p+1}\delta_{A,B})$ and
$\d(i_1,i_k)=\d(j_1,j_k)=\d(i_1\delta_{A,B},i_k\delta_{A,B})$, whence $\delta_{A,B}\in\DI_n$,
by Lemma \ref{Grpr0}, and so $\delta_{A,B}\in\ODI_n$.

Conversely, suppose that $\delta_{A,B}\in\ODI_n$. Then, in particular,
$\d(i_p,i_{p+1}) =\d(i_p\delta_{A,B},i_{p+1}\delta_{A,B})=\d(j_p,j_{p+1})$, for $1\leqslant p\leqslant k-1$, and
$\d(i_1,i_k)=\d(i_1\delta_{A,B},i_k\delta_{A,B})=\d(j_1,j_k)$, whence $\d(A)=\d(B)$.

\smallskip

Next, we prove 2. If $\d(A)=\d(Bh)$ then, by 1, $\delta_{A,Bh}\in\ODI_n$ and so,
as $k\geqslant2$ and $h|_{Bh}$ is an order-reversing partial isometry from $Bh$ onto $B$, it follows that
$\delta_{A,Bh}h|_{Bh}$ is an order-reversing partial isometry from $A$ onto $B$.

Conversely,  suppose there exists an order-reversing partial isometry $\xi$ from $A$ onto $B$. Then
$$
\xi= \left(
             \begin{array}{cccc}
               i_1 & i_2 & \cdots & i_k \\
               j_k & j_{k-1} & \cdots & j_1 \\
             \end{array}
\right)
$$
and $Bh=\{n-j_k+1<n-j_{k-1}+1<\cdots<n-j_1+1\}$, whence
$$
\delta_{A,Bh}= \left(
             \begin{array}{cccc}
               i_1 & i_2 & \cdots & i_k \\
               n-j_k+1 & n-j_{k-1}+1 & \cdots & n-j_1+1 \\
             \end{array}
\right)
=\xi h|_B \in\ODI_n
$$
and so, by 1, $\d(A)=\d(Bh)$.

\smallskip

Finally, we prove 3. First, suppose that $\d(A)=\d(Bg^{-s})$ for some $0\leqslant s\leqslant n-1$.
Then, we have $\delta_{A,Bg^{-s}}\in\ODI_n$, by 1.
Since $g^s|_{Bg^{-s}}$ is an orientation-preserving partial isometry from $Bg^{-s}$ onto $B$,
then $\delta_{A,Bg^{-s}}g^s|_{Bg^{-s}}$  is an orientation-preserving partial isometry from $A$ onto $B$.

Conversely, suppose there exists an orientation-preserving partial isometry $\xi$ from $A$ onto $B$.
If $k=2$ then
$$
\xi= \left(
             \begin{array}{cc}
               i_1 & i_2 \\
               j_1 & j_2 \\
             \end{array}
\right)=\delta_{A,B}
\quad\text{or}\quad
\xi= \left(
             \begin{array}{cc}
               i_1 & i_2 \\
               j_2 & j_1 \\
             \end{array}
\right)
$$
and so, in both cases, we get $\delta_{A,B}\in\ODI_n$, whence $\d(A)=\d(B)(=\d(Bg^{-s})$, with $s=0$), by 1.
Thus, suppose that $k>2$. Then, since an orientation-preserving restriction of an orientation-reversing permutation
must have rank less than or equal to two (cf. proof of Theorem \ref{sizeopcoc}),
there exists $0\leqslant s\leqslant n-1$ such that $\xi=g^s|_A$.
Therefore, $A=Bg^{-s}$ and so $\delta_{A,Bg^{-s}}=\delta_{A,A}=\id_A\in\ODI_n$, since any partial identity is an order-preserving partial isometry.
Hence, by 1, it follows that $\d(A)=\d(Bg^{-s})$, as required.
\end{proof}

\begin{theorem}\label{greenJ}
Let $M \in \{\ODI_n, \MDI_n, \OPDI_n\}$ and let $\alpha, \beta \in M$. Then,
$\alpha \mathscr{J} \beta$ if and only if one of the following properties is satisfied:
\begin{enumerate}
\item $|\dom(\alpha)| = |\dom(\beta)| \leqslant 1$;
\item $|\dom(\alpha)| = |\dom(\beta)| \geqslant 2$ and
$$
\d(\dom(\alpha)) = \left\{\begin{array}{ll}
                                                               \d(\dom(\beta)) & \mbox{if $M = \ODI_n$} \\
                                                               \mbox{$\d(\dom(\beta))$ or $\d(\dom(h\beta))$} & \mbox{if $M = \MDI_n$} \\
                                                              \mbox{$\d(\dom(g^s\beta))$ for some $0\leqslant s\leqslant n-1$} & \mbox{if $M = \OPDI_n$}.
                                                             \end{array}\right.
$$
\end{enumerate}
\end{theorem}
\begin{proof}
First, suppose that $\alpha\mathscr{J} \beta$ (in $M$). Then $\alpha \mathscr{J} \beta$ in $\I_n$ and so $|\dom(\alpha)| = |\dom(\beta)|$.
If $|\dom(\alpha)| = |\dom(\beta)| \leqslant 1$ there is nothing more to prove.

Thus, suppose that $|\dom(\alpha)| = |\dom(\beta)| \geqslant 2$ and let $\gamma, \lambda \in M$ be such that $\alpha = \gamma\beta\lambda$.
We can assume, without loss of generality (by considering $\gamma|_{\dom(\alpha)}$ instead of $\gamma$, if necessary),
that $\dom(\gamma) = \dom(\alpha)$.
Hence $\im(\gamma)=\dom(\beta)$.
Then, $\gamma$ is an order-preserving partial isometry from $\dom(\alpha)$ onto $\dom(\beta)$, if $M=\ODI_n$,
$\gamma$ is an order-preserving or order-reversing partial isometry from $\dom(\alpha)$ onto $\dom(\beta)$, if $M=\MDI_n$, and
$\gamma$ is an orientation-preserving partial isometry from $\dom(\alpha)$ onto $\dom(\beta)$, if $M=\OPDI_n$.
Therefore, by Lemma \ref{relJ}, we have
$$
\mbox{$\d(\dom(\alpha))=\d(\dom(\beta))$, if $M=\ODI_n$,}
$$
$$
\mbox{$\d(\dom(\alpha))=\d(\dom(\beta))$ or $\d(\dom(\alpha))=\d(\dom(\beta)h)=d(\dom(\beta)h^{-1})=\d(\dom(h\beta))$,  if $M=\MDI_n$,}
$$
and
$$
\mbox{$\d(\dom(\alpha))=\d(\dom(\beta)g^{-s})=\d(\dom(g^s\beta))$, for some $0\leqslant s\leqslant n-1$, if $M=\OPDI_n$.}
$$

\smallskip

Conversely, suppose that 1 or 2 is satisfied. If $|\dom(\alpha)| = |\dom(\beta)| \leqslant 1$ then, as $M$ contains all partial permutations of rank
less than or equal to one, it is clear that $\alpha \mathscr{J} \beta$. So, suppose that 2 holds.
Since $\dom(h\beta)=\dom(\beta)h$ and $\dom(g^s\beta)=\dom(\beta)g^{-s}$ for all $0\leqslant s\leqslant n-1$, by Lemma \ref{relJ},
we can conclude that $M$ possesses a partial transformation $\gamma$ from $\dom(\alpha)$ onto $\dom(\beta)$.
Take also $\lambda = \beta^{-1}\gamma^{-1}\alpha\in M$.
Hence, since $\gamma\beta\beta^{-1}\gamma^{-1}$ and $\gamma^{-1}\alpha\alpha^{-1}\gamma$ are idempotents, we have
$$
\gamma\beta\lambda = \gamma\beta\beta^{-1}\gamma^{-1}\alpha = \id_{\dom(\alpha)}\alpha=\alpha
\quad\text{and}\quad
\gamma^{-1}\alpha\lambda^{-1} = \gamma^{-1}\alpha\alpha^{-1}\gamma\beta = \id_{\dom(\beta)}\beta=\beta
$$
and so $\alpha \mathscr{J} \beta$, as required.
\end{proof}

\section{Generators and ranks} \label{ranks}

This section is devoted to the main result of this paper. We will determine a generating set of minimal size for each of
the monoids $\ODI_{n}$, $\MDI_{n}$ and $\OPDI_{n}$.

Let
$$
e_i=\id_{\Omega_n\setminus\{i\}}=
\begin{pmatrix} 1&\cdots&i-1&i+1&\cdots&n\\
1&\cdots&i-1&i+1&\cdots&n
\end{pmatrix}\in\DI_n,
$$
for $1\leqslant i\leqslant n$. Clearly, for $1\leqslant i,j\leqslant n$, we have $e_i^2=e_i$ and $e_ie_j=\id_{\Omega_n\setminus\{i,j\}}=e_je_i$.
More generally, for any $X\subseteq\Omega_n$, we get $\Pi_{i\in X}e_i=\id_{\Omega_n\setminus X}$.

Now, take $\alpha\in\DI_n$. Then, since the elements of $\DI_n$ are precisely the restrictions of $\D_{2n}$,
we have $\alpha=h^jg^i|_{\dom(\alpha)}$,
for some $j\in\{0,1\}$ and $i\in\{0,1,\ldots,n-1\}$.
Hence $\alpha=h^jg^i\id_{\dom(\alpha)}=h^jg^i\Pi_{k\in\Omega_n\setminus\dom(\alpha)}e_k$.
Therefore
$
\{g,h,e_1,e_2,\ldots,e_n\}
$
is a generating set of $\DI_n$. Moreover, since $e_i=g^{n-i}e_ng^i$ for all $i\in\{1,2,\ldots,n\}$,
it follows that $\{g,h,e_n\}$ is also a generating set of $\DI_n$. In fact, as $g^n=\id$, we also have $e_n=g^ie_ig^{n-i}$ and so
each set $\{g,h,e_i\}$, with $1\leqslant i\leqslant n$, generates $\DI_n$ (see \cite{Fernandes&Paulista:2022sub}).

\smallskip

Notice that $e_1,e_2,\ldots,e_n$ are elements of $\ODI_n$, $\MDI_n$ and $\OPDI_n$.
Consider the elements
$$
x=\begin{pmatrix}
1&2&\cdots&n-1\\
2&3&\cdots&n
\end{pmatrix}
\quad\text{and}\quad
y=x^{-1}=\begin{pmatrix}
2&3&\cdots&n\\
1&2&\cdots&n-1
\end{pmatrix}
$$
of $\ODI_n$ with rank $n-1$ and the elements
$$
x_i=\begin{pmatrix}
1&1+i\\
1&n-i+1
\end{pmatrix}
\quad\text{and}\quad
y_i=x_i^{-1}=\begin{pmatrix}
1&n-i+1\\
1&1+i
\end{pmatrix},
$$
for $1\leqslant i\leqslant\lfloor\frac{n-1}{2}\rfloor$, of $\ODI_n$ with rank $2$.
Observe that $\d(1,1+i)=i$, for $1\leqslant i\leqslant\lfloor\frac{n-1}{2}\rfloor$, and $\lfloor\frac{n-1}{2}\rfloor<\frac{n}{2}$.

\begin{proposition}\label{gensets}
The monoids $\ODI_n$, $\MDI_n$ and $\OPDI_n$ are generated by
$$
\{x,y,e_2,\ldots,e_{n-1},x_1,x_2,\ldots,x_{\lfloor\frac{n-1}{2}\rfloor},y_1,y_2,\ldots,y_{\lfloor\frac{n-1}{2}\rfloor}\},
$$
$$
\{h,x,e_2,\ldots,e_{\lfloor\frac{n+1}{2}\rfloor},x_1,x_2,\ldots,x_{\lfloor\frac{n-1}{2}\rfloor},y_1,y_2,\ldots,y_{\lfloor\frac{n-1}{2}\rfloor}\}
$$
and
$$
\{g,e_i,x_1,x_2,\ldots,x_{\lfloor\frac{n-1}{2}\rfloor}\},
\quad \mbox{with $1\leqslant i\leqslant n$,}
$$
respectively.
\end{proposition}
\begin{proof}
First, we show that $\{x,y,e_2,\ldots,e_{n-1},x_1,x_2,\ldots,x_{\lfloor\frac{n-1}{2}\rfloor},y_1,y_2,\ldots,y_{\lfloor\frac{n-1}{2}\rfloor}\}$ generates $\ODI_n$.

Let $M$ be the monoid generated by
$\{x,y,e_2,\ldots,e_{n-1},x_1,x_2,\ldots,x_{\lfloor\frac{n-1}{2}\rfloor},y_1,y_2,\ldots,y_{\lfloor\frac{n-1}{2}\rfloor}\}\subseteq\ODI_n$.
Then $M$ is contained in $\ODI_n$. In order to show the converse inclusion,
notice first that $e_1=yx$ and $e_n=xy$, whence $e_1,e_2,\ldots,e_n\in M$, and so $M$ contains all restrictions of each of its elements.

Next, since the elements of $\DI_n$ are the restrictions of $\D_{2n}$, then the elements of $\ODI_n$
are the order-preserving restrictions of $g^k$ and $hg^k$ for $0\leqslant k\leqslant n-1$, which are, in turn, the restrictions of
$$
g^k|_{\{1,2,\ldots,n-k\}},\quad  g^k|_{\{n-k+1,\ldots,n\}}\quad\text{and}\quad hg^k|_{\{i,j\}},
$$
with $1\leqslant i\leqslant k$ and $k+1\leqslant j\leqslant n$. Therefore, it suffices to show that these elements belong to $M$.

Notice that, if $k=0$ then $g^k|_{\{1,2,\ldots,n-k\}}$ and $g^k|_{\{n-k+1,\ldots,n\}}$ are the identity transformation and the empty transformation, respectively,
and so both belong to $M$.
So, let $1 \leqslant k \leqslant n-1$. Then, we have $g^k|_{\{1,2,\ldots,n-k\}} = x^k \in M$ and $g^k|_{\{n-k+1,\ldots,n\}} = y^{n-k} \in M$.
On the other hand, for $1\leqslant i\leqslant k$ and $k+1\leqslant j\leqslant n$, we get
$$
hg^k|_{\{i,j\}} = \left\{\begin{array}{ll}
             \prod_{\ell \in \Omega_n\setminus\{i,j\}}e_\ell & \mbox{if $i=\frac{k+1}{2}$} \\
             \prod_{\ell \in \Omega_n\setminus\{i,j\}}e_\ell x^{k-2i+1} & \mbox{if $i<\frac{k+1}{2}$} \\
             \prod_{\ell \in \Omega_n\setminus\{i,j\}}e_\ell y^{2i-k-1} & \mbox{if $i>\frac{k+1}{2}$},
           \end{array}\right.
$$
if $j-i = \frac{n}{2}$, and
$$
hg^k|_{\{i,j\}} = \left\{\begin{array}{ll}
             y^{i-1} x_{j-i} x^{k-i} & \mbox{if $j-i \leqslant \lfloor\frac{n-1}{2}\rfloor$} \\
             y^{i-1} y_{n-j+i} x^{k-i} & \mbox{if $j-i > \lfloor\frac{n-1}{2}\rfloor$},
           \end{array}\right.
$$
if $j-i \neq \frac{n}{2}$ (as usual, putting $x^0 = y^0 = \id$), and so $hg^k|_{\{i,j\}}\in M$.

Thus, we proved that $M=\ODI_n$.

\smallskip

Next, regarding the monoid $\MDI_n$, we have $\alpha=(\alpha h)h$ and $\alpha h\in \ODI_n$ for all $\alpha\in\MDI_n\setminus\ODI_n$,
which allows us to deduce that $\MDI_n$ is generated by $\ODI_n\cup\{h\}$. On the other hand,
we have $y=hxh$ and $he_ih=e_{n-i+1}$ for all $1\leqslant i\leqslant n$. Thus, we conclude that
$\{h,x,e_2,\ldots,e_{\lfloor\frac{n+1}{2}\rfloor},x_1,x_2,\ldots,x_{\lfloor\frac{n-1}{2}\rfloor},y_1,y_2,\ldots,y_{\lfloor\frac{n-1}{2}\rfloor}\}$
generates $\MDI_n$.

\smallskip

Finally, we turn our attention to the monoid $\OPDI_n$.
Let $\alpha\in\OPDI_n$. Then $\alpha\in\POPI_n$ and so, by \cite[Proposition 3.1]{Fernandes:2000},
there exist $0\leqslant k\leqslant n-1$ and $\beta\in\POI_n$ such that $\alpha=g^k\beta$.
Since $\beta=g^{n-k}\alpha\in\DI_n$, we get $\beta\in\ODI_n$.
So $\alpha=g^k\beta$, with $\beta\in\ODI_n$. Therefore, $\OPDI_n$ is generated by $\ODI_n\cup\{g\}$.
On the other hand, we have $e_j=g^{n-j}e_ng^{j}$ for all $1\leqslant j\leqslant n$,
$g^\ell x_\ell g^\ell =y_\ell$ for all $1\leqslant \ell\leqslant \lfloor\frac{n-1}{2}\rfloor$, $x=e_ng$ and $y=g^{n-1}e_n$.
Hence, $\OPDI_n$ is generated by $\{g,e_n,x_1,x_2,\ldots,x_{\lfloor\frac{n-1}{2}\rfloor}\}$.

Let $1\leqslant i\leqslant n$. Since $e_n=g^ie_ig^{n-i}$, then $\{g,e_i,x_1,x_2,\ldots,x_{\lfloor\frac{n-1}{2}\rfloor}\}$
also generates $\OPDI_n$, as required.
\end{proof}

In order to determine the ranks of these monoids, we first prove the following lemma:

\begin{lemma}\label{ranklem}
Let $1\leqslant i\leqslant\lfloor\frac{n-1}{2}\rfloor$ and
let $\gamma_1,\gamma_2,\ldots,\gamma_k,\lambda_1,\lambda_2,\ldots,\lambda_\ell$ be $k+\ell$ ($k,\ell\geqslant1$) elements of $\DI_n$
such that $x_i=\gamma_1\gamma_2\cdots\gamma_k$ and $y_i=\lambda_1\lambda_2\cdots\lambda_\ell$.
\begin{enumerate}
\item If $\gamma_1,\gamma_2,\ldots,\gamma_k,\lambda_1,\lambda_2,\ldots,\lambda_\ell\in\MDI_n$
then there exist $1\leqslant p\leqslant k$, $1\leqslant q\leqslant \ell$, $1\leqslant a<b\leqslant n$ and $1\leqslant c<d\leqslant n$
such that $\dom(\gamma_p)=\{a,b\}$, $\dom(\lambda_q)=\{c,d\}$, $b-a=i$ and $d-c=n-i$.
\item If $\gamma_1,\gamma_2,\ldots,\gamma_k\in\OPDI_n$
then there exist $1\leqslant p\leqslant k$ and $1\leqslant a<b\leqslant n$
such that $\dom(\gamma_p)=\{a,b\}$ and $b-a\in\{i,n-i\}$.
\end{enumerate}
Consequently, any generating set of $\ODI_n$, $\MDI_n$ and $\OPDI_n$ has at least $2\lfloor\frac{n-1}{2}\rfloor$, $2\lfloor\frac{n-1}{2}\rfloor$ and $\lfloor\frac{n-1}{2}\rfloor$ transformations of rank two, respectively.
\end{lemma}
\begin{proof}
First, observe that the last statement of this lemma follows immediately from the conditions 1 (notice that $\ODI_n\subseteq\MDI_n$) and 2 of the lemma, and from the fact that
$\{1,2,\ldots, \lfloor\frac{n-1}{2}\rfloor\}\cap\{n-i\mid 1\leqslant i\leqslant\lfloor\frac{n-1}{2}\rfloor\}=\emptyset$.

\smallskip

We begin by making some considerations about the elements of $\MDI_n$.

Let $\xi$ be an element of $\MDI_n$ with rank greater than or equal to $2$ and take $0\leqslant t\leqslant n-1$ such that $\xi=g^t|_{\dom(\xi)}$ or $\xi=hg^t|_{\dom(\xi)}$.

If either $\xi$ is order-reversing and $\xi=g^t|_{\dom(\xi)}$ or $\xi$ is order-preserving and $\xi=hg^t|_{\dom(\xi)}$
then $\xi$ must have rank $2$:  $\dom(\xi)=\{a<b\}$,
with $1\leqslant a\leqslant n-t<b\leqslant n$, in the first case, and $1\leqslant a\leqslant t<b\leqslant n$, in the last one.
We say that such an element $\xi$ of $\MDI_n$ is \textit{inverted}.

On the other hand, if either $\xi$ is order-preserving and $\xi=g^t|_{\dom(\xi)}$ or $\xi$ is order-reversing and $\xi=hg^t|_{\dom(\xi)}$
then, for all $a,b\in\dom(\xi)$, we have
\begin{equation}\label{case2}
\mbox{$|a\xi-b\xi|=|a-b|$.}
\end{equation}
Notice that if $a,b\in\dom(\xi)$ are such that $a<b$ then, in the first case, $1\leqslant a<b\leqslant n-t$ or $n+t+1\leqslant a<b\leqslant n$ and, in the second case,
$1\leqslant a<b\leqslant t$ or $t+1\leqslant a<b\leqslant n$. We say that such an element $\xi$ of $\MDI_n$ is \textit{non-inverted}.

Next, let $\xi_1,\xi_2,\ldots,\xi_r$ be $r$  ( $r\geqslant1$) non-inverted elements of $\MDI_n$ such that $\rank(\xi_1\xi_2\cdots\xi_r)\geqslant2$.
Then, for all $a,b\in\dom(\xi_1\xi_2\cdots\xi_r)$, by applying consecutively (\ref{case2}) to $\xi_r,\xi_{r-1},\ldots,\xi_1$, we obtain
\begin{equation}\label{morecase2}
\mbox{$|a\xi_1\xi_2\cdots\xi_r-b\xi_1\xi_2\cdots\xi_r|=|a-b|$.}
\end{equation}

Now, in order to prove 1, suppose that $\gamma_1,\gamma_2,\ldots,\gamma_k,\lambda_1,\lambda_2,\ldots,\lambda_\ell\in\MDI_n$
(keep in mind that $\gamma_1\gamma_2\cdots\gamma_k=x_i$ and $\lambda_1\lambda_2\cdots\lambda_\ell=y_i$).

If $\gamma_1,\gamma_2,\ldots,\gamma_k$ are all non-inverted elements of $\MDI_n$ then, by (\ref{morecase2}), we have
$$
n-i=|1-(n-1+i)|= |1x_i-(1+i)x_i|= |1\gamma_1\gamma_2\cdots\gamma_k-(1+i)\gamma_1\gamma_2\cdots\gamma_k| = |1-(1+i)|=i,
$$
which is a contradiction. Thus, at least one of the elements $\gamma_1,\gamma_2,\ldots,\gamma_k$ is inverted.
Let $1\leqslant p\leqslant k$ be the smallest index such that $\gamma_p$ is inverted.
Then, $\gamma_p$ has rank $2$ and, since $1\gamma_1\cdots\gamma_{p-1}, (1+i)\gamma_1\cdots\gamma_{p-1}\in\dom(\gamma_p)$,
we have $\dom(\gamma_p)=\{1\gamma_1\cdots\gamma_{p-1}, (1+i)\gamma_1\cdots\gamma_{p-1}\}$ and, by (\ref{morecase2}),
$$
|1\gamma_1\cdots\gamma_{p-1}  - (1+i)\gamma_1\cdots\gamma_{p-1}|=|1-(1+i)|=i.
$$

Similarly, if $\lambda_1,\lambda_2,\ldots,\lambda_\ell$ are all non-inverted elements of $\MDI_n$ then, by (\ref{morecase2}), we have
$$
i=|1-(1+i)|= |1y_i-(n-1+i)y_i|= |1\lambda_1\lambda_2\cdots\lambda_\ell-(n-1+i)\lambda_1\lambda_2\cdots\lambda_\ell| = |1-(n-i+1)|=n-i,
$$
which is also a contradiction. Thus, at least one of the elements $\lambda_1,\lambda_2,\ldots,\lambda_\ell$ is inverted and we may take
the smallest index $1\leqslant q\leqslant \ell$ such that $\lambda_{q}$ is inverted.
Since $1\lambda_1\cdots\lambda_{q-1}, (n+i-1)\lambda_1\cdots\lambda_{q-1}\in\dom(\lambda_{q})$ and $\lambda_{q}$ has rank $2$,
we have $\dom(\lambda_{q})=\{1\lambda_1\cdots\lambda_{q-1}, (n-i+1)\lambda_1\cdots\lambda_{q-1}\}$ and, by (\ref{morecase2}),
$$
|1\lambda_1\cdots\lambda_{q-1}  - (n-i+1)\lambda_1\cdots\lambda_{q-1}|=|1-(n-i+1)|=n-i.
$$
Therefore, we proved 1.

\smallskip

To prove 2, suppose that $\gamma_1,\gamma_2,\ldots,\gamma_k\in\OPDI_n$
(remember we have $\gamma_1\gamma_2\cdots\gamma_k=x_i$).
We begin by observing that $x_i=hg|_{\{1,1+i\}}$. Since $\d(1,1+i)=i<\frac{n}{2}$,
then $hg$ is the only extension in $\D_{2n}$ of $x_i$, by Lemma \ref{fundlemma}.
If for all $1\leqslant j\leqslant k$ there exists $0\leqslant t_j\leqslant n-1$ such that $\gamma_j=g^{t_j}|_{\dom{\gamma_j}}$, then
$x_i=g^{\sum_{j=1}^{k}t_j}|_{\{1,1+i\}}$, which contradicts the previous conclusion.
Hence, there exists $1\leqslant p\leqslant k$ such that $\gamma_p=hg^t|_{\dom(\gamma_p)}$, for some $0\leqslant t\leqslant n-1$.
Let us assume that the index $p$ is the smallest under these conditions.
Since $\gamma_p$ preserves the orientation, then $\dom(\gamma_p)=\{a,b\}$, for some $1\leqslant a\leqslant t<b\leqslant n$.
As $1\gamma_1\cdots\gamma_{p-1}, (1+i)\gamma_1\cdots\gamma_{p-1}\in\dom(\gamma_p)$,
it follows that $\dom(\gamma_p)=\{1\gamma_1\cdots\gamma_{p-1}, (1+i)\gamma_1\cdots\gamma_{p-1}\}$.

On the other hand, by the minimality of $p$,
we have $\gamma_1\cdots\gamma_{p-1}=g^s|_{\dom(\gamma_1\cdots\gamma_{p-1})}$, for some $0\leqslant s\leqslant n-1$.
Hence
$$
|1\gamma_1\cdots\gamma_{p-1}  - (1+i)\gamma_1\cdots\gamma_{p-1}|=|1g^s-(1+i)g^s| \in \{i,n-i\},
$$
as required.
\end{proof}

Recall that $\ODI_3=\POI_3$, $\MDI_3=\PODI_3$ and $\OPDI_3=\POPI_3$.
Then, the monoids $\ODI_3$, $\MDI_3$ and $\OPDI_3$ have ranks $3$, $3$ and $2$
(see \cite{Fernandes:2000,Fernandes:2001,Fernandes&Gomes&Jesus:2004}), respectively.
For $n$ greater than $3$, we have:

\begin{theorem}\label{rankth}
For $n\geqslant4$, the monoids $\ODI_n$, $\MDI_n$ and $\OPDI_n$ have ranks
$n+2\lfloor\frac{n-1}{2}\rfloor$, $2+3\lfloor\frac{n-1}{2}\rfloor$ and $2+\lfloor\frac{n-1}{2}\rfloor$,
respectively.
\end{theorem}
\begin{proof} Let $M\in \{\ODI_n, \MDI_n,\OPDI_n\}$ and let $G$ be a generating set of the monoid $M$.
Notice that the partial identities $e_1,\ldots, e_n$ belong to $M$.

\smallskip

Suppose that $M=\ODI_n$.  Then, the only permutation of $M$ is the identity and so, for $1\leqslant i\leqslant n$,
we have $e_i=\gamma_1\gamma_2\cdots\gamma_k$, for some $\gamma_1,\gamma_2,\ldots,\gamma_k\in G\setminus\{\id\}$ ($k\geqslant1$),
and so $\im(\gamma_k)=\im(e_i)=\Omega_n\setminus\{i\}$. Hence, $G$ possesses at least $n$ elements with rank $n-1$.
Thus, taking into account Lemma \ref{ranklem}, we get $|G|\geqslant n+2\lfloor\frac{n-1}{2}\rfloor$.

\smallskip

Next, suppose that $M=\MDI_n$.
Recall that $M$ has only two permutations: the identity and $h$. So, in particular, we must have $h\in G$.
Let $1\leqslant i\leqslant n$. Then, there exist
$\gamma_1,\gamma_2,\ldots,\gamma_k\in G\setminus\{\id\}$ ($k\geqslant1$) such that $e_i=\gamma_1\gamma_2\cdots\gamma_k$ and:
$\gamma_k\neq h$; or $k\geqslant2$, $\gamma_k=h$ and $\gamma_{k-1}\neq h$.
Hence, $\im(\gamma_k)=\im(e_i)=\Omega_n\setminus\{i\}$ or $\im(\gamma_{k-1})=\im(e_i)h=\Omega_n\setminus\{n-i+1\}$.
Therefore, we can conclude that $G$ possesses at least $\lfloor\frac{n+1}{2}\rfloor$ elements with rank $n-1$.
Thus, in view of Lemma \ref{ranklem},
we obtain $|G|\geqslant 1+ \lfloor\frac{n+1}{2}\rfloor+ 2\lfloor\frac{n-1}{2}\rfloor=2+3\lfloor\frac{n-1}{2}\rfloor$.

\smallskip

Finally, suppose that $M=\OPDI_n$. Since $\OPDI_{n}$ contains the permutation $g$ and a partial identity of rank $n-1$, we
can conclude that $G$ has at least one permutation and one transformation with rank $n-1$. Thus, combining with Lemma \ref{ranklem},
we get $|G|\geqslant 2+\lfloor\frac{n-1}{2}\rfloor$.

\smallskip

Since Proposition \ref{gensets} gives us generating sets of  $\ODI_n$, $\MDI_n$ and $\OPDI_n$ with
$n+2\lfloor\frac{n-1}{2}\rfloor$, $2+3\lfloor\frac{n-1}{2}\rfloor$ and $2+\lfloor\frac{n-1}{2}\rfloor$ elements, respectively,
the theorem follows.
\end{proof}

\subsection*{Acknowledgments}

We acknowledge the anonymous referees for their valuable suggestions.

The authors would like to thank Eden Santos for her help in reviewing the English of this paper.

The second and fourth authors acknowledge funding from FCT - Funda\c c\~ao para a Ci\^encia e a Tecnologia, I.P.,
under the scope of the projects UIDB/00297/2020 and UIDP/00297/2020 (NovaMath - Center for Mathematics and Applications).

\subsection*{Statements and Declarations}

Competing Interests: the authors declare no conflicts of interest.

\bigskip

\lastpage

\end{document}